\documentclass[12pt]{amsart}
\numberwithin{equation}{section}

\usepackage{mathtools}
\usepackage[margin=1.15in]{geometry}
\usepackage{xcolor}
\usepackage{hyperref}
\usepackage{fancyhdr}
\usepackage{amsmath}             
\usepackage{amssymb}           
\usepackage{amsfonts}           
\usepackage{latexsym}           
\usepackage{amsthm}              
\usepackage{bbm}
\usepackage{tikz}

\newcommand{\field}[1]{\mathbb{#1}}
\newcommand{\N}{\field{N}}

\newcommand{\R}{\field{R}}

\newcommand{\updim}{\overline{\dim}}
\newcommand{\lowdim}{\underline{\dim}}

\newcommand{\smalli}{\mathtt{i}}
\newcommand{\norm}[1]{\left|\left|#1\right|\right|}
\newcommand{\diam}{\mathrm{diam}}
\newcommand{\dima}{\dim_{\mathrm{A}}}

\newcommand{\iii}{\mathbf{i}}
\newcommand{\jjj}{\mathbf{j}}
\newcommand{\NN}{\mathcal{N}}

\DeclareMathOperator*{\essinf}{ess\,inf}

\theoremstyle{plain}
\newtheorem{thm}{Theorem}[section]
\newtheorem{lemma}[thm]{Lemma}       
\newtheorem{cor}[thm]{Corollary}      
\newtheorem{prop}[thm]{Proposition}

\theoremstyle{definition}
  
\newtheorem{quest}{Question}
\newtheorem*{quest*}{Question}

\theoremstyle{remark}
\newtheorem{huom}[thm]{Remark}   
\newtheorem{example}[thm]{Example}

\makeatletter
\@namedef{subjclassname@2020}{\textup{2020} Mathematics Subject Classification}
\makeatother

\title{Pointwise Assouad dimension for measures}
\author{Roope Anttila}
\address{Research Unit of Mathematical Sciences\\  P.O.Box 8000, FI-90014, University of Oulu, Finland }

\email{roope.anttila@oulu.fi}

\date{\today}

\begin{document}

\begin{abstract}
    We introduce a pointwise variant of the Assouad dimension for measures on metric spaces, and study its properties in relation to the global Assouad dimension. We show that, in general, the value of the pointwise Assouad dimension may differ from the global counterpart, but in many classical cases, the pointwise Assouad dimension exhibits similar exact dimensionality properties as the classical local dimension, namely it equals the global Assouad dimension almost everywhere. We also prove an explicit formula for the Assouad dimension of certain invariant measures with place dependent probabilities supported on self-conformal sets.
\end{abstract}

\maketitle
\section{Introduction}
Originally, the Assouad dimension was defined as a means to investigate embedding problems of metric spaces \cite{A}, and it is still used as an important tool in the field \cite{T}. In the past decade, the interest in the Assouad dimension has seen a substantial increase also in fractal geometry, and various Assouad-type dimensions have been developed and studied in many classical cases. The book by Fraser \cite{F} collects the recent developments in one place and provides an introduction to Assouad dimensions in fractal geometry. The Assouad dimension describes the local structure of the space by quantifying the size of the thickest parts of the space across all scales, which provides a heuristic on why it is effective in the study of embedding problems: if the space has locally thick parts, it can not be embedded into a small space.

As is usual in dimension theory, the Assouad dimension of a space is closely connected to a dual notion of dimension for measures supported on the space. For a finite Borel measure $\mu$ fully supported on a metric space $X$, this \emph{Assouad dimension of the measure} is defined by
\begin{align}\label{eq:global_assouad}
    \dima \mu=\inf\bigg\{ s>0\colon& \exists C>0,\textnormal{ s.t. for all }x\in X,\; 0<r<R,\nonumber\\
    &\frac{\mu(B(x,R))}{\mu(B(x,r))}\leq C\left(\frac{R}{r}\right)^s\bigg\},
\end{align}
where $B(x,r)$ is the closed ball with center $x$ and radius $r$. The Assouad dimension of a measure has a similar intuition behind it as the Assouad dimension of a space: it quantifies the size of the least regular parts of the measure across all scales. 

Perhaps the most important concepts in the dimension theory of measures are the \emph{upper and lower local dimensions of a measure} defined at $x\in X$ by
\begin{equation*}
    \updim_{\textnormal{loc}}(\mu,x)=\limsup_{r\to 0}\frac{\log \mu(B(x,r))}{\log r},
\end{equation*}
and
\begin{equation*}
    \lowdim_{\textnormal{loc}}(\mu,x)=\liminf_{r\to 0}\frac{\log \mu(B(x,r))}{\log r},
\end{equation*}
respectively. When the upper and lower limits agree, the limit is denoted by $\dim_{\textnormal{loc}}(\mu,x)$ and it is called the \emph{local dimension of the measure $\mu$ at $x$}. Unlike the different notions of ``global'' dimension, which are concerned with the average regularity (e.g. in the case of Hausdorff, packing and Minkowski dimensions) or extremal regularity (in the case of Assouad and lower dimensions) of the measure on its full support, these pointwise dimensions quantify the regularity of the measure around a given point. The upper and lower local dimensions can be thought of as the pointwise analogue of the Hausdorff and packing dimensions of the measure, respectively, and in this paper, we widen the theory by defining a natural pointwise analogue of the Assouad dimension. We call this dimension the \emph{pointwise Assouad dimension} of the measure and define it at $x\in X$ by
\begin{align*}
    \dima (\mu,x)=\inf\bigg\{ s>0\colon& \exists C(x)>0,\textnormal{ s.t. for all }\; 0<r<R,\\
    &\frac{\mu(B(x,R))}{\mu(B(x,r))}\leq C(x)\left(\frac{R}{r}\right)^s\bigg\}.
\end{align*}
The crucial difference to the global Assouad dimension of the measure is that the constant $C$ in the definition may depend on the point $x$. Similarly as the Assouad dimension captures information on the least regular parts of the measure across all scales, the pointwise Assouad dimension quantifies the least regular scales at a given point. The aim of this paper is to discuss the basic properties of the pointwise Assouad dimension from a fractal geometric point of view.

We note that the ideas in the definition are not entirely new, as was pointed out to us by Anders and Jana Björn as well as the anonymous referee. In \cite{BBL}, the authors consider certain ``exponent sets'', one of which corresponds to our definition of the pointwise Assouad dimension, and use them to give sharp estimates for variational $p$-capacities of annuli in metric spaces. The authors give some examples about the behaviour of these exponent sets, however, most of the measures they consider are absolutely continuous, which by the Lebesgue differentiation theorem can have ``non-trivial'' (that is different from the dimension of the Lebesgue measure) pointwise Assouad dimension only in a set of measure zero. In contrast, the measures we study are singular and, as we will see, their pointwise Assouad dimensions have ``non-trivial'' behaviour almost everywhere. Therefore, we believe that the present article is a welcome contribution to the theory from a fractal geometric point of view, and that the results in this paper and \cite{BBL} complement each other quite well.

\subsection{Main results and the structure of the paper}
The structure of the paper is as follows. We begin by establishing some notation and recalling basic results concerning Assouad dimensions of sets and measures in Section \ref{sec:prelim}. In Section \ref{sec:pw_assouad} we discuss some basic properties of the pointwise Assouad dimension and its relations to various existing notions of dimension. In particular, we observe that the pointwise Assouad dimension is always bounded from above by the global Assouad dimension, and that the inequality can be strict. To contrast this, we devote the rest of the paper to the study of the cases where the maximal pointwise Assouad dimension coincides with the global one.
In fact, we show that in many classical constructions we have an \emph{exact dimensionality property}, meaning
\begin{equation}\label{eq:exact_dimensionality}
    \dima(\mu,x)= \dima\mu,
\end{equation}
for $\mu$-almost every $x$. This is analogous to the classical exact dimensionality property for the local dimension, which the measure is said to satisfy if $\dim_{\mathrm{loc}}(\mu,x)=\dim_{\mathrm{H}}\mu$, for $\mu$-almost every $x$.  We start with the general setting of quasi-Bernoulli measures supported on strongly separated self-conformal sets in Section \ref{sec:qb-meas},  and in Theorem \ref{thm:quasi-bernoulli-exact-assouad} prove the exact dimensionality property (\ref{eq:exact_dimensionality}) for these measures. The results are complemented in in Section \ref{sec:invmeas}, where we provide an explicit formula for the Assouad dimension of certain invariant measures for place dependent probabilities in Theorem \ref{thm:pw_ssc_formula}. These place dependent invariant measures satisfy the assumptions of Section \ref{sec:qb-meas} so as a corollary we obtain the almost sure formula for the pointwise Assouad dimension as well. In the final Section \ref{sec:ssm} we are interested in self-similar and self-affine measures. In Theorems \ref{thm:osc_formula} and \ref{thm:bm_exact_dimensionality}, we prove the exact dimensionality property (\ref{eq:exact_dimensionality}) for self-similar measures satisfying the open set condition and self-affine measures on certain Bedford-McMullen carpets, respectively.

\section{Preliminaries}\label{sec:prelim}
Unless stated otherwise, we assume that $(X,d)$ is a metric space, with no additional structure. Since we assume the metric $d$ to be fixed, we omit it from the notation and refer to $(X,d)$ simply as $X$. Unless stated otherwise, a measure always refers to a finite Borel measure fully supported on $X$ and when needed, we denote the support of $\mu$ by $\mathrm{supp}(\mu)$. If $f:X\to Y$ is a map from $X$ to another metric space $Y$, we denote the \emph{pushforward of the measure $\mu$} under the map $f$ by $f_*\mu\coloneqq\mu\circ f^{-1}$.  For constants $C$, we sometimes use the convention $C(\cdots)$, if we want to emphasize the dependence of $C$ on the quantities inside the parentheses.

\subsection{Assouad dimension of sets and connection to weak tangents}
We are mainly focused on dimensions of measures in this paper, but to place the results in a wider context, we recall some results concerning the Assouad dimensions of sets. The \emph{Assouad dimension of a set} $F\subset X$ is defined by
\begin{align*}
    \dima F=\inf\bigg\{ s>0\colon& \exists C>0,\textnormal{ s.t. for all }x\in F,\; 0<r<R,\\
    &N_r(B(x,R)\cap F)\leq \left(\frac{R}{r}\right)^s\bigg\},
\end{align*}
where $N_r(E)$ denotes the smallest number of open balls of diameter $r$ needed to cover the set $E\subset X$.  A convenient way to bound the Assouad dimension of a set from below is given by the weak tangent approach. Recall that a map $T\colon X\to X$ is a \emph{similarity} if there exists $c>0$, such that $d(T(x),T(y))=c d(x,y)$, for all $x,y\in X$. The constant $c$ is called the \emph{similarity ratio (of $T$)}. 

For simplicity we give the definition of weak tangents when $X\subset \R^d$ and make a brief remark that they can be defined in complete metric spaces using pointed convergence in the Gromov-Hausdorff distance \cite{KL,MT}.  A closed set $F\subset B(0,1)$ is said to be a \emph{weak tangent} of a compact set $X\subset \R^d$ if there is a sequence of similarities $T_n\colon \R^d\to\R^d$, such that
\begin{equation*}
    T_n(X)\cap B(0,1)\to F,
\end{equation*}
in the Hausdorff distance. The collection of weak tangents of $X$ is denoted by $\mathrm{Tan}(X)$. The following proposition gives an easy way to bound the Assouad dimension from below. For the proof in the general setting see e.g. \cite[Proposition 6.1.5]{MT}.

\begin{prop}\label{prop:weaktangent}
If $X\subset \R^d$ is compact, then $\dim_{\mathrm{A}} X\geq \dim_{\mathrm{A}}F$, for all $F\in\mathrm{Tan}(X)$.
\end{prop}

\subsection{Assouad dimension of measures and the doubling property}\label{subsec:assouad}
Let us now turn our attention to the dimensions of measures. When referring to (\ref{eq:global_assouad}) we sometimes use the term global Assouad dimension to avoid ambiguity with the pointwise variant. Originally, the global Assouad dimension of a measure was called the upper regularity dimension in \cite{KLV}, but due to the intimate connections between this notion of dimension and the Assouad dimension for sets, the term Assouad dimension of a measure is now widely used. A simple volume argument implies that for a measure $\mu$ fully supported on a metric space $X$, we have the inequality $\dima X\leq \dima\mu$. Moreover, in \cite{VK, LS} it was shown that
\begin{equation*}
    \dima X=\inf\{\dima\mu\colon \mu\textnormal{ is a measure fully supported on }X\},
\end{equation*}
which further supports the current terminology.

The Assouad dimension of a measure has two important properties. First of all, it characterizes the \emph{doubling property}, which the measure $\mu$ is said to satisfy if there is a constant $C\geq 1$, such that for any $x\in X$, $r>0$, we have
\begin{equation}\label{eq:doubling}
    \mu(B(x,2r))\leq C\mu(B(x,r)).
\end{equation}
Measures that satisfy (\ref{eq:doubling}) are called \emph{doubling measures} and it is a simple exercise to show that a measure has finite Assouad dimension if and only if it is doubling \cite[Lemma 4.1.1]{F}.

Secondly, the Assouad dimension is ``the greatest of all dimensions'' \cite{F}, that is
\begin{equation*}
    \dim_{\mathrm{H}}\mu\leq \dim_{\mathrm{P}}\mu\leq\dima\mu,
\end{equation*}
where $\dim_{\mathrm{H}}\mu\coloneqq\essinf_{x\in X}\lowdim_{\mathrm{loc}}(\mu,x)$ and $\dim_{\mathrm{P}}\mu\coloneqq\essinf_{x\in X}\updim_{\mathrm{loc}}(\mu,x)$ are the \emph{Hausdorff dimension} and the \emph{packing dimension} of the measure $\mu$ respectively. As we will see in the next section, the pointwise Assouad dimension enjoys properties which are analogous to the two properties mentioned.

\section{Pointwise Assouad dimension}\label{sec:pw_assouad}
In this section, we discuss the basic properties of the pointwise Assouad dimension of measures. It turns out that we have a correspondence between measures with the \emph{pointwise doubling property} and those with finite pointwise Assouad dimension. A measure $\mu$ is said to be \emph{pointwise doubling at $x\in X$}, if there is a constant $C(x)\geq 1$, such that
\begin{equation*}
    \mu(B(x,2r))\leq C(x)\mu(B(x,r)).
\end{equation*}
We also refer to \cite{BBL} for some discussion on the pointwise doubling property.
The following proposition collects some of the basic properties of the pointwise Assouad dimension. Note that the properties (1) and (2) are the pointwise analogues for the basic properties of the global Assouad dimension discussed in Subsection \ref{subsec:assouad}.
\begin{prop}\label{prop:basic_properties}
Let $\mu$ be a Borel measure fully supported on a metric space $X$. Then for any $x\in X$,
\begin{enumerate}
    \item[(1)] $\dim_{\mathrm{A}}(\mu,x)$ is finite if and only if $\mu$ is pointwise doubling at $x$,
    \item[(2)] $\lowdim_{\textnormal{loc}}(\mu,x)\leq\updim_{\textnormal{loc}}(\mu,x)\leq \dim_{\mathrm{A}}(\mu,x)\leq \dima\mu$,
    \item[(3)] if $\mu$ has an atom at $x$, then $\dim_{\mathrm{A}}(\mu,x)=0$.
\end{enumerate}
\end{prop}
\begin{proof}
    Claim (1). is a trivial modification of \cite[Lemma 4.1.1]{F}. 
    
    For (2), note that the first and last inequalities follow straight from the definitions, so it suffices to prove the middle inequality. Fix $x\in X$, and let $s> \dim_{\mathrm{A}}(\mu,x)$ be arbitrary. Then by definition, there is a constant $C$ depending only on $x$, such that for all $0<r<R$,
    \begin{equation*}
        \frac{\mu(B(x,r))}{\mu(B(x,R))}\geq C\left(\frac{r}{R}\right)^s.
    \end{equation*}
    In particular, by fixing $R$ we see that $\mu(B(x,r))\geq c r^s$,
    where $c=\frac{C\mu(B(x,R))}{R^s}$. Taking logarithms, dividing by $\log r$ and taking $r\to 0$ shows that $\updim_{\textnormal{loc}}(\mu,x)\leq s$. Since $s>\dim_{\mathrm{A}}(\mu,x)$ was arbitrary, this finishes the proof.

    For (3), assume that $\mu$ has an atom at $x\in X$. Let $0<r<R$, and note that
    \begin{equation*}
        \frac{\mu(B(x,R))}{\mu(B(x,r))}\leq \frac{\mu(X)}{\mu(\{x\})}=\frac{\mu(X)}{\mu(\{x\})}\left(\frac{R}{r}\right)^0.
    \end{equation*}
    Since the constant $\frac{\mu(X)}{\mu(\{x\})}$ depends only on $x$, we have $\dim_{\mathrm{A}}(\mu,x)\leq 0$, which is enough to prove the claim.
\end{proof}

\begin{huom}
One can define the \emph{pointwise lower dimension of $\mu$ at $x$} analogously to the lower dimension of a measure as
\begin{align*}
    \dim_{\mathrm{L}}(\mu,x)=\sup\bigg\{s>0\colon &\exists C(x)>0,\textnormal{ s.t. }\forall0<r<R,\\
    &\frac{\mu(B(x,R))}{\mu(B(x,r))}\geq C(x)\left(\frac{R}{r}\right)^s\bigg\}.
\end{align*}
Then Proposition \ref{prop:basic_properties}(2) is strengthened to 
\begin{equation*}
    \dim_{\mathrm{L}}\mu\leq\dim_{\mathrm{L}}(\mu,x)\leq\lowdim_{\textnormal{loc}}(\mu,x)\leq\updim_{\textnormal{loc}}(\mu,x)\leq \dim_{\mathrm{A}}(\mu,x)\leq\dima\mu,
\end{equation*}
for all $x\in X$. We will not, however, pursue the study of the pointwise lower dimension any further in this paper.
\end{huom}

\subsection{Relationships to other dimensions}

Next we investigate the relationships between the pointwise Assouad dimension and other common notions of dimension in fractal geometry. We will focus on the global Assouad dimension, the packing dimension, and the upper Minkowski dimension of measures and the Assouad dimension of the support of the measure.

As Proposition \ref{prop:basic_properties} shows, the global Assouad dimension provides an upper bound for the pointwise Assouad dimension at every point. The natural question that arises is if the converse holds at some point, i.e. is it true that $\sup_{x\in X}\dima(\mu,x)=\dima\mu$. It turns out that generally speaking this is not the case, even in compact spaces. The following is an example of a non-doubling measure, which is pointwise doubling at all points of its support. By Proposition \ref{prop:basic_properties}(1) and the analogous fact for the global Assouad dimension, we see that this measure has finite pointwise Assouad dimension at all points, but infinite global Assouad dimension.

\begin{example}\label{ex:loc_not_glob_doubling}
    Let $x_n=2^{-n}$, and let $\mu=\sum_{n=0}^{\infty}(3^{-n}\delta_{-x_n}+2^{-n}\delta_{x_n})$, where $\delta_x$ denotes the point mass centered at $x$. Clearly the measure is a finite Borel measure fully supported on the set $X=\{0\}\cup\bigcup_{n=0}^{\infty}\{x_n,-x_n\}$. By considering $y_k=-x_k$, and $r_k=2^{-k}$, it follows by a simple calculation that
    \begin{equation*}
        \frac{\mu(B(y_k,2r_k))}{\mu(B(y_k,r_k))}\geq\frac{\sum_{n=k-1}^{\infty}3^{-n}+\sum_{n=k}^{\infty}2^{-n}}{\sum_{n=k-1}^{\infty}3^{-n}}=1+\left(\frac{3}{2}\right)^{k-2} \to \infty,
    \end{equation*}
    as $k\to\infty$, which shows that $\mu$ is not doubling.
    
    The fact that $\mu$ is pointwise doubling at $X\setminus\{0\}$ follows from properties (1) and (3) of Proposition \ref{prop:basic_properties}, and a standard calculation shows that for any  $2^{-k}\leq r < 2^{-k+1}$, we have
    \begin{equation*}
        \frac{\mu(B(0,2r))}{\mu(B(0,r))}\leq \frac{2\sum_{n=k-2}^{\infty}2^{-n}}{\sum_{n=k}^{\infty}2^{-n}}\leq\frac{2^{4-k}}{2^{1-k}}=8,
    \end{equation*}
    which shows that $\mu$ is pointwise doubling at $0$.
\end{example}

The upper Minkowski dimension of a measure was introduced in \cite{KFF} and it is defined by
\begin{align*}
    \updim_{\mathrm{M}}\mu=\inf\{s>0\colon &\textnormal{there exists a constant }c>0\textnormal{ such that}\\
    &\mu(B(x,r))\geq cr^s,  \textnormal{ for all }x\in \text{supp}(\mu)\textnormal{ and }0<r<1\}.
\end{align*}
In \cite[Proposition 4.1]{KFF} it was shown that
\begin{equation}\label{eq:general_relationship}
    \dim_{\mathrm{P}}\mu\leq \updim_{\mathrm{M}}\mu\leq \dim_{\mathrm{A}}\mu.
\end{equation}
By definition of $\dim_{\mathrm{P}}\mu$ and Proposition \ref{prop:basic_properties}(2) we have the general relationship $\dim_{\mathrm{P}}\mu\leq\dim_{\mathrm{A}}(\mu,x)$, for $\mu$-almost every $x$. Recalling (\ref{eq:general_relationship}), it is natural to ask if a similar inequality holds for the upper Minkowski dimension in some direction. This turns out not to be the case. For some self-affine measures on Bedford-McMullen carpets (see Section \ref{sec:ssm} for the definitions), we have $\updim_{\mathrm{M}}\mu \leq\dima(\mu,x)$, for $\mu$-almost all $x$. Example \ref{ex:bm_example} provides an explicit example satisfying this property. However, it is also possible to have $\updim_{\mathrm{M}}\mu >\dima(\mu,x)$ for all $x$, as the following example shows.

\begin{example}
    Let $\mu=\sum_{n=0}^{\infty}(3^{-n}\delta_{-2^{-n}}+2^{-n}\delta_{2^{-n}})$ be the measure of Example \ref{ex:loc_not_glob_doubling}. Let us first show that $\updim_{\mathrm{M}}\mu\geq \frac{\log 3}{\log 2}$. Let $s<\frac{\log 3}{\log 2}$, $x_n=-2^{-n}$, and $r_n=2^{-(n+2)}$. Then
    \begin{equation*}
        \mu(B(x_n,r_n))=3^{-n}=r_n^{\frac{n\log 3}{(n+2)\log 2}}< r_n^{s},
    \end{equation*}
    for all large enough $n$. Since $r_n\to 0$ with $n$, this implies that $\updim_{\mathrm{M}}\mu\geq s$, and taking $s\to\frac{\log 3}{\log 2}$ gives the claim.
    
   Since $\mu$ has an atom at every $x\in X\setminus\{0\}$, by Proposition \ref{prop:basic_properties}(3), $\dim_{\mathrm{A}}(\mu,x)=0<\updim_{\mathrm{M}}\mu$. At the origin, a simple calculation shows that for any $2^{-L}< r \leq 2^{-L+1}$ and $2^{-N-1}\leq R < 2^{-N}$, we have
    \begin{align*}
        \frac{\mu(B(0,R))}{\mu(B(0,r))}\leq \frac{2\sum_{n=N}^{\infty}2^{-n}}{\sum_{n=L}^{\infty}2^{-n}}=\frac{2^{2-N}}{2^{1-L}}\leq 8\left(\frac{R}{r}\right),
    \end{align*}
    and therefore $\dim_{\mathrm{A}}(\mu,0)\leq1<\frac{\log3}{\log 2}\leq \updim_{\mathrm{M}}\mu$.
\end{example}

Finally, a natural question to ask is if the Assouad dimension of the support of a measure is a lower bound for the pointwise Assouad dimension of the measure, as it is for the global one. Our next example shows that this is also not generally the case, in fact, there are measures supported on sets of full Assouad dimension, which have $0$ pointwise Assouad dimension at all points. The example is original, but builds on a construction by Le Donne and Rajala \cite[Example 2.20]{LR}.

\begin{example}
Let $x_{n,k}=2^{-2^n}+k4^{-2^n}$ and let $X=\{0\}\cup\bigcup_{n=1}^{\infty}\bigcup_{k=0}^{n-1}\{x_{n,k}\}$. Define the measure $\mu$ as 
\begin{equation*}
    \mu=\sum_{n=1}^{\infty}\sum_{k=0}^{n-1}\frac{2^{-n}}{n}\delta_{x_{n,k}}.
\end{equation*}
It is straightforward to show that $\mu$ is a finite doubling measure fully supported on $X$. We show that $\dim_{\mathrm{A}}X=1$, and $\dim_{\mathrm{A}}(\mu,x)=0$, for every $x\in X$.

To show that $\dim_{\mathrm{A}}X= 1$, by Proposition \ref{prop:weaktangent} it is enough to show that $[0,1]$ is a weak tangent for the set $X$. For each $n\in\N$ define a similarity $T_n:\R\to \R$ by
\begin{equation*}
    T_n(x)=n^{-1}4^{2^n}(x-2^{-2^n}),
\end{equation*}
and note that
\begin{equation*}
    T_n(X)\cap[0,1]=\bigcup_{k=0}^{n-1}\Big\{\frac{k}{n}\Big\}\to[0,1],
\end{equation*}
in the Hausdorff distance as $n\to\infty$, that is $[0,1]$ is a weak tangent to $X$.

Next we show that $\dim_{\mathrm{A}}(\mu,x)=0$, for every $x\in X$. Note that every point $x\in X\setminus\{0\}$ is an atom so by Proposition \ref{prop:basic_properties}(3), $\dim_{\mathrm{A}}(\mu,x)=0$, so we only need to consider the case $x=0$. Fix $0<r<R<1$, and choose numbers $L,N\in\N$, such that $2^{-2^L}< r \leq 2^{-2^{L-1}}$ and $2^{-2^{N+1}}\leq R < 2^{-2^{N}}$. Clearly $[0,x_{L+1,L}]\subset B(0,r)$ and $B(x,R)\subset [0,x_{N,N-1}]$, so we have
\begin{align*}
    \frac{\mu(B(0,R))}{\mu(B(0,r))}\leq\frac{\mu([0,x_{N,N-1}])}{\mu([0,x_{L+1,L}])}\leq 2^{L-N+1}\leq 2\frac{\log r}{\log R}.
\end{align*}
Note that for any $s>0$, the function $\phi(t)=t^s\log t$ is decreasing for  $0<t<e^{-\frac{1}{s}}$, so for all $0<r<R<e^{-\frac{1}{s}}$ we have
\begin{equation*}
    \frac{\mu(B(0,R))}{\mu(B(0,r))}\leq 2\left(\frac{R}{r}\right)^s.
\end{equation*}
Since this holds for arbitrary $s>0$, we have $\dim_{\mathrm{A}}(\mu,0)=0$.
\end{example}

\section{Quasi-Bernoulli measures on self-conformal sets}\label{sec:qb-meas}
As was observed in Example \ref{ex:loc_not_glob_doubling}, a strict inequality is certainly possible in $\sup_{x}\dima(\mu,x)\leq \dima\mu$, but it turns out that in many natural cases where the measure has some kind of rigid structure, the pointwise Assouad dimension coincides with the global variant almost everywhere. We start by proving the exact dimensionality property (\ref{eq:exact_dimensionality}) for the most general case of this paper in this section, and work our way down to more specific classes of measures in Sections \ref{sec:invmeas} and \ref{sec:ssm}. For the convenience of the readers who are familiar with the definitions, we give the statement of the main result of this section first, and define the necessary concepts after that. The main result we prove at the end of this section is the following.

\begin{thm}\label{thm:quasi-bernoulli-exact-assouad}
    If $\mu$ is a quasi-Bernoulli measure fully supported on a self-conformal set $F$ satisfying the strong separation condition, then
    \begin{equation*}
        \dim_{\mathrm{A}}(\mu,x)=\dim_{\mathrm{A}}\mu<\infty,
    \end{equation*}
    for $\mu$-almost every $x\in F$.
\end{thm}
\begin{huom}
    In this generality, it is not possible to obtain an explicit formula for the Assouad dimension, however, to complement the result, in Section \ref{sec:invmeas} we provide an example class of measures which satisfy the assumptions of Theorem \ref{thm:quasi-bernoulli-exact-assouad}, and calculate their Assouad dimension explicitly.
\end{huom}

Let us start by recalling some basics of iterated function systems. Let $\Lambda$ be a finite index set, and associate to each $i\in\Lambda$ a contraction map $\varphi_i$ from a compact subset of $\R^d$ to itself. The collection $\{\varphi_i\}_{i\in\Lambda}$ is known as an \emph{iterated function system (IFS)}. By a foundational result of Hutchinson \cite{H}, every IFS has a unique compact and non-empty invariant set $F$ satisfying
\begin{equation*}
    F=\bigcup_{i\in\Lambda}\varphi_i(F),
\end{equation*}
called the \emph{limit set of the IFS}. To make the study easier, one often imposes some restrictions on the defining maps. In this section we will concentrate on the class of \emph{quasi-Bernoulli measures} supported on \emph{self-conformal sets}. In addition to the conformality assumption, which we define later, we require that the IFS satisfies the \emph{strong separation condition (SSC)}, namely we assume that for any distinct $i,j\in\Lambda$, we have $\varphi_i(F)\cap \varphi_j(F)=\emptyset$.

When studying limit sets of iterated function systems, it is often useful to consider a symbolic representation of the IFS. Let $\Sigma=\{(i_1,i_2,\ldots)\colon i_k\in\Lambda\}$ denote the set of infinite sequences of the symbols in $\Lambda$. We call $\Sigma$ the \emph{symbolic space} and members of $\Sigma$ \emph{(infinite) words}. For an integer $n$, let $\Sigma_n=\{(i_1,i_2,\ldots,i_n)\colon i_k\in\Lambda\}$ be the set of \emph{finite words of length $n$} and let $\Sigma_*=\bigcup_{n\in\N}\Sigma_n\cup\{\emptyset\}$ denote the set of all finite words of any length. For any $\mathbf{i}\in\Sigma$, let $\mathbf{i}|_0=\emptyset$ denote the empty word. We use the abbreviation $\mathbf{i}=(i_1,i_2,\ldots)$ for a fixed element of $\Sigma$ and the same notation $\mathbf{i}=(i_1,\ldots,i_n)$ for elements of $\Sigma_n$, but the meaning will be clear from the context. For $\mathbf{i}=(i_1,\ldots,i_n)\in\Sigma_n$, let $\mathbf{i}^-=(i_1,\ldots,i_{n-1})$ denote the finite word obtained by dropping the last element of $\mathbf{i}$. If $\mathbf{i}\in\Sigma$, we write $\mathbf{i}|_n=(i_1,\ldots,i_n)\in\Sigma_n$ for the projection of $\mathbf{i}$ onto the first $n$ coordinates. For $\mathbf{i}\in\Sigma_n$, the cylinder $[\mathbf{i}]\subset \Sigma$ is defined to be the set of all infinite words in $\Sigma$ whose first $n$ letters are the letters of $\mathbf{i}$. In some proofs, we use for $\mathbf{i}\in\Sigma$ and $\mathbf{j}\in\Sigma_*$ the notation $\mathbf{j}\sqsubset\mathbf{i} $, to mean that the word $\mathbf{i}$ contains the word $\mathbf{j}$ as a substring.

For the contractions $\varphi_i$ we abbreviate
\begin{equation*}
    \varphi_{\mathbf{i}|_n}=\varphi_{i_1}\circ\ldots\circ \varphi_{i_n}.
\end{equation*}
Recall that there is a natural correspondence between the symbolic space $\Sigma$ and the limit set $F$ by the coding map $\pi:\Sigma\to F$ defined by
\begin{equation}\label{eq:code_map}
    \{\pi(\mathbf{i})\}=\bigcap_{n=1}^{\infty}\varphi_{\mathbf{i}|_n}(F).
\end{equation}
When $F$ satisfies the SSC, this map is a bijection.

\subsection{Self-conformal sets}\label{subsec:self-conformal}

Next we define self-conformal sets which act as a support for the measures we study in this section. Recall that an IFS $\{\varphi_i\}_{i\in\Lambda}$ on $\R^d$ is called \emph{self-conformal} if it satisfies the following assumptions:

\begin{enumerate}
    \item[(C1)] There is a set $\Omega\subset \R^d$, which is open, bounded and connected, and a compact set $X\subset\Omega$ with non-empty interior, such that
    \begin{equation*}
        \varphi_{i}(X)\subset X,
    \end{equation*}
    for all $i\in\Lambda$.
    \item[(C2)] For each $i\in\Lambda$, the map $\varphi_i$ is a $C^{1+\varepsilon}$-diffeomorphism, and $\varphi_i\colon \Omega\to\Omega$ is conformal. That is, for all $x\in\Omega$, the linear map $\varphi_i'(x)$ is a similarity. In particular, for every $y\in \Omega$, we have
    \begin{equation*}
        |\varphi_i'(x)y|= |\varphi_i'(x)||y|,
    \end{equation*}
    where $|\varphi_i'(x)|$ denotes the operator norm of the linear map $\varphi_i'(x)$.
\end{enumerate}
The use of the bounded open set $\Omega$ here is essential since contractive conformal maps defined on whole $\R^d$ are in fact similarities. The limit set $F$ of an IFS satisfying (C1) and (C2) is called a \emph{self-conformal set}. In the following we let $||\varphi'_{\mathbf{i}}||=\sup_{x\in \Omega}|\varphi'_{\mathbf{i}}(x)|$. It follows from the fact that each $\varphi_i$ is a contraction, that $||\varphi'_{\mathbf{i}}||<1$, for all $\mathbf{i}\in\Sigma_*$, and that for a fixed $\mathbf{i}\in\Sigma$, $||\varphi'_{\mathbf{i}|_n}||$ is strictly decreasing in $n$. Let us recall some key lemmas for the proof of Theorem \ref{thm:quasi-bernoulli-exact-assouad}.

\begin{lemma}\label{lemma:BDP}
There are constants $C>1,\delta>0$, such that the self-conformal set $F$ satisfies the following.

\begin{enumerate}
    \item For all $\mathbf{i}\in\Sigma_*$ and $x,y\in \Omega$, we have $|\varphi_{\mathbf{i}}'(x)|\leq C|\varphi_{\mathbf{i}}'(y)|$.
    
    \item For any $x,y,z\in F$, with $|x-y|\leq \delta$, we have
    \begin{equation*}
        C^{-1}|\varphi_{\mathbf{i}}'(z)|\leq\frac{|\varphi_{\mathbf{i}}(x)-\varphi_{\mathbf{i}}(y)|}{|x-y|}\leq C|\varphi_{\mathbf{i}}'(z)|,
    \end{equation*}
    for all $\mathbf{i}\in\Sigma_*$.

    \item For all $\mathbf{i}\in\Sigma_*$,
    \begin{equation*}
        C^{-1}\norm{\varphi'_{\mathbf{i}}}\leq\diam(\varphi_{\mathbf{i}}(F)) \leq C\norm{\varphi'_{\mathbf{i}}}.
    \end{equation*}
\end{enumerate}
\end{lemma}
The first property in Lemma \ref{lemma:BDP} is commonly called the \emph{Bounded Distortion Property (BDP)} and it originates in \cite{MU}. Property (ii) is a special case of \cite[Lemma 2.3]{Fan} and property (iii) is proved in \cite{MU}.

\begin{lemma}\label{lemma:chain}
For all $\mathbf{i},\mathbf{j}\in\Sigma_*$, we have
\begin{equation*}
    C^{-1}||\varphi'_{\mathbf{i}}||\cdot||\varphi_{\mathbf{j}}'||\leq||\varphi_{\mathbf{i}\mathbf{j}}'||\leq ||\varphi_{\mathbf{i}}'||\cdot||\varphi_{\mathbf{j}}'||,
\end{equation*}
where $C>1$ is the constant of Lemma \ref{lemma:BDP}.
\end{lemma}
\begin{proof}
Using the chain rule, and the conformality of the IFS, it is easy to see that for all $x\in F$ we have
\begin{equation*}
    |\varphi_{\mathbf{i}\mathbf{j}}'(x)|= |(\varphi_{\mathbf{i}}\circ\varphi_{\mathbf{j}})'(x)|
    =|\varphi_{\mathbf{i}}'(\varphi_{\mathbf{j}}(x))\cdot\varphi_{\mathbf{j}}'(x)|=|\varphi_{\mathbf{i}}'(\varphi_{\mathbf{j}}(x))|\cdot|\varphi_{\mathbf{j}}'(x)|.
\end{equation*}
Applying Lemma \ref{lemma:BDP} we get that for all $y\in F$
\begin{equation*}
    C^{-1}|\varphi_{\mathbf{i}}'(y)|\cdot|\varphi_{\mathbf{j}}'(x)|\leq |\varphi_{\mathbf{i}\mathbf{j}}'(x)| \leq |\varphi_{\mathbf{i}}'(\varphi_{\mathbf{j}}(x))|\cdot|\varphi_{\mathbf{j}}'(x)|\leq ||\varphi_{\mathbf{i}}'||\cdot||\varphi_{\mathbf{j}}'||.
\end{equation*}
The result follows by taking suprema.
\end{proof}
\begin{huom}
Let $\mathbf{i}\in\Sigma$ be $n$-periodic. Notice that, by applying the previous lemma iteratively, we have
\begin{equation*}
    C^{-k}||\varphi'_{\mathbf{i}|_n}||^k\leq||\varphi'_{\mathbf{i}|_{kn}}||\leq ||\varphi'_{\mathbf{i}|_n}||^k.
\end{equation*}
The exponential growth of the distortion in the lower bound is a problem in Section \ref{sec:invmeas} when we want to establish a lower bound for the Assouad dimension of the measure we investigate. The following lemma provides a precise estimate for this purpose.
\end{huom}

\begin{lemma}\label{lemma:power_equality}
If $x= \pi(\mathbf{i})$, where $\mathbf{i}\in\Sigma$ is $n$-periodic for some $n\in\N$, then for any $k\in\N$ we have
\begin{equation*}
    |\varphi_{\mathbf{i}|_{kn}}'(x)|=|\varphi_{\mathbf{i}|_{n}}'(x)|^k.
\end{equation*}
\end{lemma}
\begin{proof}
Let $\mathbf{i}\in \Sigma$ be $n$-periodic, and let $x=\pi(\mathbf{i})$. By the definition of $\pi$, this implies that 
\begin{equation}\label{eq:x_invariance}
    \varphi_{\mathbf{i}|_n}(x)=x.
\end{equation}
Let $k\in\N$. Using the chain rule, (\ref{eq:x_invariance}) and the conformality of the IFS we find that
\begin{align*}
    |\varphi_{\mathbf{i}|_{kn}}'(x)|&=|((\varphi_{i_1}\circ\ldots\circ\varphi_{i_n})\circ\underset{k\text{ times}}{\ldots}\circ(\varphi_{i_1}\circ\ldots\circ\varphi_{i_n}))'(x)|\\
    &=|(\varphi_{i_1}\circ\ldots\circ \varphi_{i_n})'(x)\cdot\underset{k\text{ times}}{\ldots}\cdot(\varphi_{i_1}\circ\ldots\circ \varphi_{i_n})'(x)|\\
    &=|\varphi_{\mathbf{i}|_n}'(x)^k|=|\varphi_{\mathbf{i}|_n}'(x)|^k.
\end{align*}
\end{proof}

\subsection{Quasi-Bernoulli measures}
A probability measure $\nu$ on $\Sigma$ is called \emph{quasi-Bernoulli} if there exists a uniform constant $C\geq 1$, such that for all $\mathbf{i},\mathbf{j}\in\Sigma_*$, we have
\begin{equation*}
    C^{-1}\nu([\mathbf{i}])\nu([\mathbf{j}])\leq \nu([\mathbf{i}\mathbf{j}])\leq C\nu([\mathbf{i}])\nu([\mathbf{j}]),
\end{equation*}
where here and hereafter $\mathbf{i}\mathbf{j}$ denotes the concatenation of the finite words $\mathbf{i}$ and $\mathbf{j}$. Note the similarity to Lemma \ref{lemma:chain}. If $C$ can be taken to equal $1$, then the measure is called a \emph{Bernoulli measure}.

To simplify notation, from hereafter we write $A\lesssim B$ to mean that $A$ is bounded from above by $B$ multiplied by a uniform constant. Similarly, we say that $A\gtrsim B$, if $B\lesssim A$ and $A\approx B$ if $B\lesssim A\lesssim B$.

Recall that two measures $\mu$ and $\nu$ are said to be \emph{equivalent} if $\mu(A)=0$ if and only if $\nu(A)=0$. In the following $\sigma\colon \Sigma\to\Sigma$ denotes the \emph{left-shift} defined by $\sigma(\iii)=i_2i_2\ldots$. Also recall that a measure $\nu$ on $\Sigma$ is \emph{ergodic} if either $\nu(A)=0$ or $\nu(A)=1$, for all $\sigma$-invariant sets $A\subset \Sigma$.

For the rest of this section, let $\NN\subset \Sigma$ denote the set of infinite words, which contain all finite words as a substring, that is
\begin{equation*}
    \NN = \{\iii\in\Sigma\colon \jjj\sqsubset\iii,\text{ for all }\jjj\in\Sigma_*\}.
\end{equation*}
The following lemma is simple, but crucial to the proof of Theorem \ref{thm:quasi-bernoulli-exact-assouad}.
\begin{lemma}\label{lemma:quasi-normal}
If $\nu$ is a quasi-Bernoulli measure, then $\nu(\Sigma\setminus\NN)=0$.
\end{lemma}
\begin{proof}
It is well known that if $\nu$ is a quasi-Bernoulli measure, then the measure $\tilde{\nu}$ obtained as a weak-$*$ accumulation point of the sequence
\begin{equation*}
    \tilde{\nu}_n\coloneqq\frac{1}{n}\sum_{j=0}^{n-1}\nu\circ\sigma^{-j},
\end{equation*}
is a $\sigma$-invariant and ergodic quasi-Bernoulli measure, which is equivalent with $\nu$ \cite{Heur}. Therefore, we may assume without loss of generality that $\nu$ is $\sigma$-invariant and ergodic. Now for every $\mathbf{j}\in\Sigma_*$, we see by applying Birkhoff's ergodic theorem, that
\begin{align*}
    \lim_{n\to\infty}\frac{1}{n}\sum_{i=0}^{n}\chi_{[\mathbf{j}]}(\sigma^n\mathbf{i})=\int_{\Sigma}\chi_{[\mathbf{j}]}d\nu = \nu([\mathbf{j}])>0,
\end{align*}
for $\nu$-almost every $\mathbf{i}\in\Sigma$, where $\chi_{[\mathbf{j}]}$ denotes the indicator function of the set $[\mathbf{j}]$. In particular, this implies that for almost every $\mathbf{i}$, there is $n\in\N$, such that $\chi_{[\mathbf{j}]}(\sigma^n\mathbf{i})=1$, that is $\mathbf{j}\sqsubset\mathbf{i}$. Let $\Sigma_{\mathbf{j}}=\{\mathbf{i}\in\Sigma\colon \mathbf{j}\sqsubset\mathbf{i}\}$, so by the previous $\nu(\Sigma\setminus\Sigma_{\mathbf{j}})=0$. By definition of $\mathcal{N}$, we have
\begin{equation*}
    \nu(\Sigma\setminus\mathcal{N})=\nu\left(\Sigma\setminus\bigcap_{\mathbf{j}\in\Sigma_*}\Sigma_{\mathbf{j}}\right)=\nu\left(\bigcup_{\mathbf{j}\in\Sigma_*}\Sigma\setminus\Sigma_{\mathbf{j}}\right)\leq \sum_{\mathbf{j}\in\Sigma_*}\nu(\Sigma\setminus\Sigma_{\mathbf{j}})=0.
\end{equation*}
\end{proof}
We say that a measure $\mu$ supported on a self-conformal set $F$ is quasi-Bernoulli if it is the projection of a quasi-Bernoulli measure $\nu$ supported on $\Sigma$ under the natural projection $\pi\colon\Sigma\to F$ defined as in (\ref{eq:code_map}). Next we show that the quasi-Bernoulli measures supported on self-conformal sets satisfying the SSC are doubling, which in particular implies that the Assouad dimensions of these measures are finite. After that, we prove the main theorem of this section, Theorem \ref{thm:quasi-bernoulli-exact-assouad}.
\begin{prop}\label{prop:doubling}
If $\mu$ is a quasi-Bernoulli measure fully supported on a self-conformal set $F$ satisfying the strong separation condition, then $\mu$ is doubling.
\end{prop}

\begin{proof}
Let $\delta = \min_{i\ne j}d(\varphi_i(F),\varphi_j(F))$, which is positive since $F$ is strongly separated. Fix an integer $k$ satisfying $\max_{\mathbf{i}\in\Sigma_k}\norm{\varphi_{\mathbf{i}}'}< \frac{\delta}{2C^4}$, where $C$ is the maximum of the constant given by Lemma \ref{lemma:BDP}. Finally, let $c=\min_{\mathbf{i}\in\Sigma_{k+1}}\nu([\mathbf{i}])$.

Since $F$ is compact and $\mu$ is fully supported on $F$, it is easy to see that it suffices to consider only uniformly small values of $r>0$.
Therefore let 
\begin{equation*}
0<r<\min\{||\varphi_{\mathbf{i}|_k}'||\colon \mathbf{i}\in\Sigma_k\}.
\end{equation*}
Note that the right hand side is positive since $X$ is compact so this is possible. Also fix $x\in F$, let $\mathbf{i}\in\Sigma$ be such that $\pi(\mathbf{i})=x$, and choose $n\in\N$ as the unique integer satisfying $C||\varphi_{\mathbf{i}|_n}'||< r \leq C||\varphi_{\mathbf{i}|_{n-1}}'||$. This immediately implies that $\varphi_{\mathbf{i}|_n}(F)\subset B(x,r)$. Note that by the assumption on $r$ we have
\begin{equation*}
    ||\varphi_{\mathbf{i}|_n}'|| <||\varphi_{\mathbf{i}|_k}'||.
\end{equation*}
so in particular $n>k$. For any $l\in\N$ and $i\in\Lambda$ let $\mathbf{i}|_li$ denote the word $(i_1,i_2,\ldots, i_l,i)$ and notice that by the strong separation condition and Lemma \ref{lemma:BDP}(1)-(3), we have for all large enough $l\in\N$ and $i\ne j$ that
\begin{equation}\label{eq:ssc}
    d(\varphi_{\mathbf{i}_l i}(F),\varphi_{\mathbf{i}_l j}(F))\coloneqq\inf_{\substack{x\in \varphi_{\mathbf{i}_l i}(F) \\y\in \varphi_{\mathbf{i}_l j}(F)}}|x-y|\geq \frac{\delta}{C^2}\cdot \diam(\varphi_{\mathbf{i}|_l}(F))\geq \frac{\delta}{C^3}||\varphi_{\mathbf{i}|_l}'||.
\end{equation}
Now using Lemma \ref{lemma:chain}, we have for every $y\in B(x,2r)$
\begin{align*}
    d(y,\varphi_{\mathbf{i}|_{n-k-1}}(F))&\leq 2r< \frac{\delta}{C^4\max_{\mathbf{i}\in\Sigma_k}||\varphi_{\mathbf{i}}'||}r\leq \frac{\delta}{C^3\max_{\mathbf{i}\in\Sigma_k}\norm{\varphi_{\mathbf{i}}'}}||\varphi_{\mathbf{i}|_{n-1}}'||\\
    &\leq\frac{\delta}{C^3}\frac{||\varphi_{\mathbf{i}|_{n-1}}'||}{||\varphi_{\sigma^{n-k-1}\mathbf{i}|_{k}}'||}\leq \frac{\delta}{C^3} ||\varphi_{\mathbf{i}|_{n-k-1}}'||,
\end{align*}
in particular, combining this with estimate (\ref{eq:ssc}), we have $B(x,2r)\cap F\subset \varphi_{\mathbf{i}|_{n-k-1}}(F)$. Therefore, using the quasi-Bernoulli property, we have
\begin{align*}
    \frac{\mu(B(x,2r))}{\mu(B(x,r))}&\leq \frac{\mu(\varphi_{\mathbf{i}|_{n-k-1}}(F))}{\mu(\varphi_{\mathbf{i}|_{n}}(F)))}=\frac{\nu([\mathbf{i}|_{n-k-1}])}{\nu([\mathbf{i}|_{n}])}\lesssim \frac{1}{\nu([\sigma^{n-k-1}\mathbf{i}|_{k+1}])}\leq \frac{1}{c}.
\end{align*}
Since the upper bound is independent of $x$ and $r$, the claim follows.
\end{proof}

\begin{proof}[Proof of Theorem \ref{thm:quasi-bernoulli-exact-assouad}]
It follows from Proposition \ref{prop:doubling} that $\dim_{\mathrm{A}}\mu$ is finite. Let $s<\dim_{\mathrm{A}}\mu$ and $c>0$.
Now there is a point $y\in F$ and radii $0<r<R$, satisfying
\begin{equation*}
    \frac{\mu(B(y,R))}{\mu(B(y,r))}> c\left(\frac{R}{r}\right)^s.
\end{equation*}
Let $\mathbf{i}\in \mathcal{N}$, $x=\pi(\mathbf{i})$ and let $\mathbf{j}\in\Sigma$, such that $\pi(\mathbf{j})=y$.  Now choose $k,n\in\N$ as the unique integers which satisfy
\begin{equation*}
    ||\varphi'_{\mathbf{j}|_{n+1}}||\leq R < ||\varphi'_{\mathbf{j}|_{n}}|| \text{, and }\quad ||\varphi'_{\mathbf{j}|_{k}}||< r \leq ||\varphi'_{\mathbf{j}|_{k-1}}||.
\end{equation*} 
Then $\varphi_{\mathbf{j}|_k}(F)\subset B(y,Cr)$ and $B(y,\frac{\delta}{2C}R)\cap F\subset \varphi_{\mathbf{j}|_n}(F)$, where $C$ is the constant of Lemma \ref{lemma:BDP}. Using the quasi-Bernoulli property and the fact that $\mu$ is doubling, we get that
\begin{equation*}
    c\left(\frac{R}{r}\right)^s<\frac{\mu(B(y,R))}{\mu(B(y,r))}\lesssim \frac{\mu(\varphi_{\mathbf{j}|_n}(F))}{\mu(\varphi_{\mathbf{j}|_k}(F))}=\frac{\nu([\mathbf{j}|_n])}{\nu([\mathbf{j}|_k])}\lesssim \nu([\sigma^{n}\mathbf{j}|_{k-n}])^{-1}.
\end{equation*}
Now since $\mathbf{i}\in \mathcal{N}$, there is an index $l\in\N$, such that $\sigma^l\mathbf{i}|_{k-n}=\sigma^{n}\mathbf{j}|_{k-n}$. Let $R'=||\varphi'_{\mathbf{i}|_{l}}||$ and $r'= ||\varphi'_{\mathbf{i}|_{l+n}}||$, and observe that by Lemma \ref{lemma:chain},
\begin{equation*}
    \frac{R'}{r'}=\frac{||\varphi'_{\mathbf{i}|_{l}}||}{||\varphi'_{\mathbf{i}|_{l+n}}||}\lesssim \frac{1}{||\varphi'_{\sigma^l\mathbf{i}|_{n}}||}=\frac{1}{||\varphi'_{\sigma^{n}\mathbf{j}|_{k-n}}||}\lesssim\frac{ ||\varphi'_{\mathbf{j}|_{n}}||}{||\varphi'_{\mathbf{j}|_{k}}||}=\frac{R}{r}.
\end{equation*}
Again, it is easy to see that $\varphi_{\mathbf{i}|_{l}}(F)\subset B(x,CR')$ and $B(x,\frac{\delta}{2C}r')\cap F\subset \varphi_{\mathbf{i}|_{l+n}}(F)$, so using the doubling and quasi-Bernoulli properties of $\mu$, we see that
\begin{align*}
    \frac{\mu(B(x,R'))}{\mu(B(x,r'))}&\gtrsim\frac{\mu(\varphi_{\mathbf{i}|_{l}}(F))}{\mu(\varphi_{\mathbf{i}|_{l+k-n}}(F))}=\frac{\nu([\mathbf{i}|_{l}])}{\nu([\mathbf{i}|_{l+k-n}])}\gtrsim \nu([\sigma^l\mathbf{i}|_{k-n}])^{-1}\\
    &=\nu([\sigma^{n}\mathbf{j}|_{k-n}])^{-1} \gtrsim c\left(\frac{R}{r}\right)^s\gtrsim c\left(\frac{R'}{r'}\right)^s.
\end{align*}
This shows that $\dim_{\mathrm{A}}(\mu,x)\geq s$, and taking $s\to\dim_{\mathrm{A}}\mu$ gives $\dim_{\mathrm{A}}(\mu,x)\geq \dim_{\mathrm{A}}\mu$. Since this holds for all $\mathbf{i}\in\mathcal{N}$, the claim follows from Lemma \ref{lemma:quasi-normal}.
\end{proof}

\section{Measures with place dependent probabilities}\label{sec:invmeas}
In this section, we study the class of \emph{invariant measures with place dependent probabilities} supported on strongly separated self-conformal sets. The results of Section \ref{sec:qb-meas} show that these measures are doubling and that their pointwise Assouad dimension coincides with the global Assouad dimension at almost every point. Our main result of this section, Theorem \ref{thm:pw_ssc_formula}, complements these results by giving an explicit formula for their Assouad dimension. To our knowledge, the formula has not been previously established in the literature. Let us begin by defining our setting.

\subsection{Place dependent invariant measures}
We assume that our IFS $\{\varphi_i\}_{i\in\Lambda}$ is self-conformal, that is it satisfies (C1) and (C2). In contrast to the case of self-conformal measures where we assign a uniform measure $p_i$ on the set $\varphi_i(F)$, now we allow the mass concentration to depend continuously on the point, that is we choose for each $i\in\Lambda$ a Hölder continuous function $p_i\colon X\to (0,1)$, which satisfy $\sum_{i\in\Lambda}p_i(x)\equiv 1$ and consider the probability  measures satisfying the equation
\begin{equation*}
    \int f(x)d\mu(x)=\sum_{i\in\Lambda}\int p_i(x)f\circ \varphi_i(x)d\mu(x),
\end{equation*}
for $f\in C(X)$ where here and hereafter $C(X)$ denotes the set of continuous real valued functions on $X$. We define the \emph{Ruelle operator} $T\colon C(X)\to C(X)$ by
\begin{equation*}
    (Tf)(x)=\sum_{i\in\Lambda}p_i(x)f\circ \varphi_i(x),
\end{equation*}
and let $T^*\colon M(X)\to M(X)$ denote the adjoint operator, where $M(X)$ is the set of Borel probability measures on $X$. Recall that for $\nu\in M(X)$, $T^*\nu$ is given by
\begin{equation*}
    T^*\nu(B)=\sum_{i\in\Lambda}\int_{\varphi_i^{-1}(B)}p_i(x)d\nu(x),
\end{equation*}
for all Borel subsets $B\subset X$. Barnsley et al. \cite{Barn} as well as Fan and Lau \cite{Fan} have studied the measures which are invariant under $T$ in a setting which is more general than ours. The next proposition, which is vital to this section, is a special case of \cite[Theorem 1.1]{Fan} or \cite[Theorem 2.1]{Barn} and we refer to the mentioned papers for the proof.

\begin{prop}\label{prop:invmeas}
Let $F$ be a self-conformal set satisfying the SSC and $p_i\colon X\to (0,1)$ be Hölder continuous for every $i\in\Lambda$. Then there is a unique Borel probability measure $\mu$ satisfying
\begin{equation*}
    T^*\mu=\mu.
\end{equation*}
Furthermore, for every $f\in C(X)$, $T^nf$ converges uniformly to the constant $\int f(x)d \mu(x)$.
\end{prop}
We call the measure $\mu$ an \emph{invariant measure with place dependent probabilities}, which we shorten to just \emph{invariant measure} for the remainder of this section. As was the case with self-similar measures and Bernoulli measures on the corresponding code space, there is also a natural correspondence between the invariant measure $\mu$ and a Gibbs measure on the code space. Let us define some useful notation for this section. For $\mathbf{i}\in\Sigma$ we slightly abuse notation by writing $p_i(\mathbf{i})\coloneqq p_i(\pi(\mathbf{i}))$ and $\varphi_i(\mathbf{i})\coloneqq\varphi_i(\pi(\mathbf{i}))$, where $\pi\colon \Sigma\to F$ is the natural projection given by (\ref{eq:code_map}). For $\mathbf{i}\in\Sigma$ and $n\in\N$ we let
\begin{equation*}
    p_{\mathbf{i}|_n}(\mathbf{i}) =\prod_{k=1}^{n}p_{i_k}(\sigma^{k-1}\mathbf{i}).
\end{equation*}
Denote by $P(\Sigma)\subset\Sigma$ the set of periodic points of $\Sigma$. For $\mathbf{i}\in P(\Sigma)$ with period of length $n$, we let
\begin{equation*}
    \overline{p}_{\mathbf{i}}=p_{\mathbf{i}|_{n}}(\mathbf{i}),
\end{equation*}
and
\begin{equation*}
    |\varphi'_{\mathbf{i}}|=|\varphi'_{\mathbf{i}|_{n}}(\mathbf{i})|.
\end{equation*}
The following lemma is well known consequence of the Ruelle-Perron-Frobenius theorem \cite{Bow}, and it follows from Proposition 1.3 and Theorem 1.6 of \cite{Fan}.

\begin{lemma}\label{lemma:gibbs}
There exists a unique $\sigma$-invariant probability measure $\nu$ on $\Sigma$, and a constant $C>1$ such that for any $x\in F$, $\mathbf{i}\in\Sigma$ and $n\in\N$, we have
\begin{equation*}
    C^{-1}p_{\mathbf{i}|_n}(\mathbf{i})\leq \nu([\mathbf{i}|_n])\leq Cp_{\mathbf{i}|_n}(\mathbf{i}).
\end{equation*}
Furthermore, $\nu$ is quasi-Bernoulli and we have $\mu=\pi_*\nu$, where $\mu$ is the measure of Proposition \ref{prop:invmeas}. 
\end{lemma}
The previous lemma together with Proposition \ref{prop:doubling} immediately imply that the measure $\mu$ is doubling. We are now ready to state the main result of this section. Recall that $P(\Sigma)$ denotes the set of periodic words in $\Sigma$.

\begin{thm}\label{thm:pw_ssc_formula}
Let $\mu$ be a place dependent invariant measure fully supported on a self-conformal set $F$, which satisfies the SSC. Then
\begin{equation*}
    \dim_{\mathrm{A}}\mu=\sup_{\mathbf{i}\in P(\Sigma)}\frac{\log \overline{p}_{\mathbf{i}}}{\log |\varphi_{\mathbf{i}}'|}.
\end{equation*}
\end{thm}
\begin{proof}
Let us start with the upper bound. For the rest of the proof let $s=\sup_{\mathbf{i}\in P(\Sigma)}\frac{\log \overline{p}_{\mathbf{i}}}{\log |\varphi_{\mathbf{i}}'|}$. Let $x\in F$ and $\mathbf{i}\in\Sigma$, such that $\pi(\mathbf{i})=x$. Let $0<r<R$ and choose integers $k$ and $n$ which satisfy
\begin{equation*}
    |\varphi'_{\mathbf{i}|_{n+1}}(x)|\leq R < |\varphi'_{\mathbf{i}|_{n}}(x)|,\text{ and} \quad |\varphi'_{\mathbf{i}|_{k}}(x)|< r \leq |\varphi'_{\mathbf{i}|_{k-1}}(x)|.
\end{equation*}
This immediately implies that $\varphi_{\mathbf{i}|_k}(F)\subset B(x,Cr)$, where $C$ is the constant of Lemma \ref{lemma:BDP}. As before, let $\delta = \min_{i\ne j}\{d(\varphi_i(F),\varphi_j(F))\colon i\ne j\}$. Then it is also easy to see that $B(x,\frac{\delta}{2C}R)\cap F\subset \varphi_{\mathbf{i}|_n}(F)$. Let us set $\mathbf{j}=(i_{k-n+1},i_{k-n+2},\ldots,i_k,i_{k-n+1},i_{k-n+2},\ldots,i_k,\ldots)\in P(\Sigma)$. Note that Lemma \ref{lemma:gibbs} shows that
\begin{equation*}
    \prod_{j=k-n+1}^kp_{i_j}(\sigma^{j-1}\mathbf{i})\gtrsim \prod_{l=1}^np_{j_l}(\sigma^{l-1}\mathbf{j}).
\end{equation*}
Using Proposition \ref{prop:doubling} to conclude that $\mu$ is doubling and Lemmas \ref{lemma:BDP}(1) and \ref{lemma:chain}, we get
\begin{align*}
    \frac{\mu(B(x,R))}{\mu(B(x,r))}&\lesssim \frac{\mu(B(x,\frac{\delta}{2C}R))}{\mu(B(x,Cr))}\leq \frac{\mu(\varphi_{\mathbf{i}|_n}(F))}{\mu(\varphi_{\mathbf{i}|_k}(F))}\lesssim \frac{p_{\mathbf{i}|_n}(\mathbf{i})}{p_{\mathbf{i}|_k}(\mathbf{i})}\\
    &=\frac{\prod_{j=1}^np_{i_j}(\sigma^{j-1}\mathbf{i})}{\prod_{j=1}^kp_{i_j}(\sigma^{j-1}\mathbf{i})} =\left(\prod_{j=k-n+1}^kp_{i_j}(\sigma^{j-1}\mathbf{i})\right)^{-1}\\
    &\lesssim\left(\prod_{l=1}^np_{j_l}(\sigma^{l-1}\mathbf{j})\right)^{-1}= |\varphi'_{\mathbf{j}|_n}(\mathbf{j})|^{-\frac{\log \prod_{l=1}^np_{j_l}(\sigma^{l-1}\mathbf{j})}{\log |\varphi'_{\mathbf{j}|_n}(\mathbf{j})|}}\\
    &\leq|\varphi'_{\mathbf{j}|_n}(\mathbf{j})|^{-s}\lesssim  |\varphi'_{\mathbf{i}|_{k-n-1}}(x)|^{-s}\lesssim\left(\frac{|\varphi'_{\mathbf{i}|_{n}}(x)|}{|\varphi'_{\mathbf{i}|_{k}}(x)|}\right)^s\lesssim\left(\frac{R}{r}\right)^s.
\end{align*}
This shows that $\dim_{\mathrm{A}}\mu\leq s$, for any $x\in F$.

For the lower bound, let $t<s$, and choose $\mathbf{i}\in P(\Sigma)$, such that
\begin{equation*}
    \frac{\log \overline{p}_{\mathbf{i}}}{\log |\varphi'_{\mathbf{i}|_{n}}(x)|}\geq t,
\end{equation*}
where $x=\pi(\mathbf{i})$ and $n$ is the period of $\mathbf{i}$. For every $k\in\N$ let $r_k=|\varphi_{\mathbf{i}_{kn}}'(x)|=|\varphi_{\mathbf{i}_{n}}'(x)|^k$, where the second equality follows from Lemma \ref{lemma:power_equality}. Using the SSC and Proposition \ref{prop:doubling} we get
\begin{align*}
    \mu(B(x,r_k))&\lesssim \mu(\varphi_{\mathbf{i}|_{kn}}(F))=\prod_{j=1}^{kn}p_{i_j}(\sigma^{j-1}\mathbf{i})\\
    &=\overline{p}_{\mathbf{i}}^k=r_k^\frac{k\log \overline{p}_{\mathbf{i}}}{k\log |\varphi_{\mathbf{i}|_{n}}'(x)|}\leq r_k^t.
\end{align*}
Taking logarithms and limits shows us that
\begin{equation*}
    \updim_{\text{loc}}(\mu,x)\geq t,
\end{equation*}
so in particular by Proposition \ref{prop:basic_properties}(2),  $\dim_{\mathrm{A}}\mu\geq t$. Taking $t\to s$ finishes the proof.
\end{proof}
Using the fact that the measure $\mu$ is quasi-Bernoulli, Theorem \ref{thm:quasi-bernoulli-exact-assouad} gives the following immediate corollary.
\begin{cor}\label{cor:invmeas-pw}
If $\mu$ is a place dependent invariant measure fully supported on a self-conformal set $F$ satisfying the SSC, then
\begin{equation*}
    \dim_{\mathrm{A}}(\mu,x)=\sup_{\mathbf{i}\in P(\Sigma)}\frac{\log \Bar{p}_{\mathbf{i}}}{\log|\varphi'_{\mathbf{i}}|},
\end{equation*}
for $\mu$-almost every $x\in F$
\end{cor}

\begin{huom}\label{huom:ssm}
By \cite[Theorem 2.4]{FH}, the Assouad dimension of a self-similar measure $\mu$ under the SSC is given by the formula
\begin{equation}\label{eq:ssc_formula}
    \dim_{\mathrm{A}}\mu=\max_{i\in\Lambda}\frac{\log p_i}{\log r_i},
\end{equation}
see Section \ref{sec:ssm} for definitions. Our Theorem \ref{thm:pw_ssc_formula} can be viewed as a generalization of this result. Indeed, an IFS consisting of similarities $\varphi_i$ with similarity ratios $r_i$ is a self-conformal IFS, with $|\varphi'_i(x)|=r_i$, for all $x\in F$. Moreover, when each $p_i(x)\equiv p_i$, the assumptions of Theorem \ref{thm:pw_ssc_formula}, and we have
\begin{equation*}
    \dim_{\mathrm{A}}\mu=\sup_{\mathbf{i}\in P(\Sigma)}\frac{\log \overline{p}_{\mathbf{i}}}{\log |\varphi'_{\mathbf{i}}|}=\sup_{\mathbf{i}\in \Sigma_*}\frac{\log p_{\mathbf{i}}}{\log r_{\mathbf{i}}}=\max_{i\in\Lambda}\frac{\log p_i}{\log r_i}.
\end{equation*}
This together with Corollary \ref{cor:invmeas-pw} gives the following slightly stronger version of \cite[Theorem 2.4]{FH}: If $\mu$ is a self-similar measure satisfying the SSC, then
\begin{equation*}
    \dim_{\mathrm{A}}(\mu,x)=\max_{i\in\Lambda}\frac{\log p_i}{\log r_i}=\dim_{\mathrm{A}}\mu,
\end{equation*}
for $\mu$-almost every $x\in F$.
\end{huom}

\section{Self-similar and self-affine measures}\label{sec:ssm}
In this section we will concentrate on two important IFS constructions: self-similar and self-affine measures. In the main results of this section, Theorems \ref{thm:osc_formula} and \ref{thm:bm_exact_dimensionality}, we establish the exact dimensionality property (\ref{eq:exact_dimensionality}), for doubling self-similar measures satisfying the open set condition, and for some self-affine measures supported on Bedford-McMullen sponges.

Let $F$ be a limit set of an IFS, as defined in the beginning of Section \ref{sec:qb-meas}, and attach to each $i\in\Lambda$ a probability $p_i\in(0,1)$, such that $\sum_{i\in\Lambda}p_i=1$. Recall that by \cite{H}, there exists a unique Borel probability measure $\mu$ fully supported on $F$ satisfying
\begin{equation*}
    \mu=\sum_{i\in\Lambda}p_i\mu\circ \varphi_i^{-1}.
\end{equation*}
When the contractions $\varphi_i$ are similarities or affine maps, $F$ is called a \emph{self-similar} or a \emph{self-affine set}, respectively, and $\mu$ is called a \emph{self-similar} or a \emph{self-affine measure}. If the maps $\varphi_i$ are similarities, we denote their similarity ratios by $r_i\in(0,1)$. In all of the proofs, we assume that $\diam(F)=1$, which does not result in loss of generality, since re-scaling the set does not affect its geometry.

Self-similar and self-affine sets and measures are perhaps the most important prototypical examples of fractal sets and measures. These classes have been well studied in the past decades, and substantial progress has been made in understanding their dimensional properties. See for example \cite{Hoch} for recent developments in the self-similar case and \cite{BHR, HR} for the self-affine case. To make their study easier, it is usual to impose some sort of separation conditions on the defining maps. The most common separation conditions are the SSC (see Section \ref{sec:qb-meas}) as well as the \emph{open set condition (OSC)}, which the set $F$ is said to satisfy if there exists an open set $U\subset \R^d$, such that $\varphi_i(U)\subset U$ for all $i\in\Lambda$ and $\varphi_i(U)\cap \varphi_j(U)=\emptyset$ for $i\ne j$. We say that a self-similar measure fully supported on a self-similar set $F$ satisfies the SSC if $F$ does and similarly for the OSC.

\subsection{Self-similar measures and the open set condition}
The doubling properties of self-similar measures are quite well studied. The fact that self-similar measures satisfying the SSC are doubling follows, for example, from Proposition \ref{prop:doubling}, and as mentioned in Remark \ref{huom:ssm}, their Assouad dimension was explicitly computed by Fraser and Howroyd \cite{FH}, and it is given by the formula (\ref{eq:ssc_formula}). Slightly surprisingly, relaxing the SSC to the OSC changes the situation dramatically. Yung \cite{Yung} provides examples of self-similar sets satisfying the OSC for which (1) only the canonical self-similar measure is doubling, (2) all self-similar measures are doubling, (3) the measures are doubling for some (non-canonical) but not all choices of the weights $p_i$. In particular this shows that the Assouad dimension of self-similar measures satisfying the OSC can in many cases be infinite. Still, it is an interesting question to study the Assouad dimension of \emph{doubling} self-similar measures which do not satisfy the SSC. In the main theorem of this section, we show that if a self-similar measure satisfying the OSC is doubling, then the Assouad dimension is given by the natural formula (\ref{eq:ssc_formula}). Furthermore, we show that the pointwise Assouad dimension agrees with the global Assouad dimension almost everywhere, obtaining a stronger version of Remark \ref{huom:ssm}.

In this section, we use the notation
\begin{equation*}
    c_{\mathbf{i}|_n}=\prod_{k=1}^nc_{i_k},
\end{equation*}
for any parameters $c_i$, with $i\in\Lambda$. Recall that we may construct a Bernoulli measure $\nu$ on $\Sigma$ by setting for all $\mathbf{i}\in\Sigma_n$,
\begin{equation*}
    \nu([\mathbf{i}])=p_{\mathbf{i}},
\end{equation*}
and extending this to the whole space $\Sigma$ in the usual way. There is a natural correspondence between the Bernoulli measure $\nu$ and the self-similar measure on $F$, namely
\begin{equation}\label{eq:bernoulli}
    \mu=\pi_*\nu,
\end{equation}
where $\pi\colon \Sigma\to F$ is the coding map given by (\ref{eq:code_map}). The proof of the next theorem, which is our main result of this section, builds on ideas of \cite{Yung} and \cite{FH}.

\begin{thm}\label{thm:osc_formula}
Let $\mu$ be a self-similar measure satisfying the OSC. If $\mu$ is doubling, then
\begin{equation*}
    \dim_{\mathrm{A}}\mu=\max_{i\in\Lambda}\frac{\log p_i}{\log r_i},
\end{equation*}
and for $\mu$-almost every $x\in F$, we have
\begin{equation*}
    \dim_{\mathrm{A}}(\mu,x)=\dima\mu.
\end{equation*}

\end{thm}
\begin{proof}
Let $s=\max_{i\in\Lambda}\frac{\log p_i}{\log r_i}$. We start by showing that $\dima\mu\leq s$. Let $x\in F$, $0<r<R<1$ and let $\iii\in\Sigma$, be a (not necessarily unique) word satisfying $\pi(\iii)=x$. Choose integers $k$ and $l$, such that $r_{\mathbf{i}|_k}\leq R < r_{\mathbf{i}|_{k-1}}$ and $\quad r_{\mathbf{i}|_{l+1}}< r \leq r_{\mathbf{i}|_{l}}$.
We may assume from this point on that $l>k$, since otherwise $\frac{R}{r}$ would be bounded from above by a uniform constant, which does not bother us. Now $\varphi_{\mathtt{i}_{l+1}}(F)\subset B(x,r)$, so in particular $\mu(B(x,r))\geq p_{\mathbf{i}|_{l+1}}$. Define $\Lambda_{x,R}=\{\mathbf{j}\in\Sigma_*\colon r_{\mathbf{j}}\leq R < r_{\mathbf{j}^{-}}, \,d(x,\varphi_{\mathbf{j}}(F))\leq R\}$. Note that the OSC implies that there is a constant $M\geq 1$ independent of $x$ and $R$, such that
\begin{equation*}
    \#\Lambda_{x,R}\leq M.
\end{equation*}
For the simple proof of this see \cite[Proposition 1.5.8]{Kig}. By definition, for every $\mathbf{j}\in\Lambda_{x,R}$ we have $\diam(\varphi_{\mathbf{j}}(F))=r_{\mathbf{j}}\leq R < r_{\mathbf{i}|_{k-1}}$,
and since $x\in \varphi_{\mathbf{i}|_{k-1}}(F)$, this implies that $d(\varphi_{\mathbf{i}|_{k-1}}(F),\varphi_{\mathbf{j}}(F))\leq R < r_{\mathbf{i}|_{k-1}}$.
Combining these estimates we see that $\varphi_{\mathbf{j}}(F)\subset B(\varphi_{\mathbf{i}|_{k-1}}(F),2r_{\mathbf{i}|_{k-1}})$, where $B(\varphi_{\mathbf{i}|_{k-1}}(F),2r_{\mathbf{i}|_{k-1}})$ denotes the open $2r_{\mathbf{i}|_{k-1}}$-neighbourhood of the set $\varphi_{\mathbf{i}|_{k-1}}(F)$. Therefore, since $\mu$ is doubling, we may apply Theorem 1.1 of \cite{Yung} to see that there is a constant $C>0$, such that
\begin{equation*}
    p_{\mathbf{j}}\leq Cp_{\mathbf{i}|_{k-1}}
\end{equation*}
holds independently from $\mathbf{j}$ and $\mathbf{i}$.
Furthermore, it is clear that
\begin{equation*}
    B(x,R)\cap F\subset\bigcup_{\mathbf{j}\in\Lambda_{x,R}}\varphi_{\mathbf{j}}(F),
\end{equation*}
so we may estimate
\begin{align*}
    \frac{\mu(B(x,R))}{\mu(B(x,r))}&\leq \frac{\sum_{\mathbf{j}\in\Lambda_{x,R}}p_{\mathbf{j}}}{p_{\mathbf{i}|_{l+1}}}\leq MC\frac{p_{\mathbf{i}|_{k-1}}}{p_{\mathbf{i}|_{l+1}}}=\frac{MC}{p_{i_k}p_{i_l}}\frac{p_{\mathbf{i}|_k}}{p_{\mathbf{i}|_l}}\\
    &\leq \frac{MC}{p_{\min}^2}\left(p_{i_{l-k+1}}p_{i_{l-k+2}}\cdots p_{i_l}\right)^{-1}\\
    &\leq \frac{MC}{p_{\min}^2}\left(r_{i_{l-k+1}}^{\frac{\log p_{i_{k-l+1}}}{\log r_{i_{k-l+1}}}}r_{i_{l-k+2}}^{\frac{\log p_{i_{k-l+2}}}{\log r_{i_{k-l+2}}}}\cdots r_{i_l}^{\frac{\log p_{i_{l}}}{\log r_{i_{l}}}}\right)^{-1}\\
    &\leq\frac{MC}{p_{\min}^2}\left(\frac{r_{\mathbf{i}|_{k}}}{r_{\mathbf{i}|_l}}\right)^{s}
    \leq  \frac{MC}{p_{\min}^2}\left(\frac{R}{r}\right)^{s},
\end{align*}
which is enough to show that $\dim_{\mathrm{A}}\mu\leq s$.

To finish the proof, it is enough to show that the lower bound holds for the pointwise Assouad dimension at almost every point. For this, let $i\in\Lambda$ be the index maximizing $\frac{\log p_i}{\log r_i}$ and define $\mathcal{N}_n=\{\mathbf{i}\in\Sigma\colon (i,\ldots,i)\sqsubset \mathbf{i}\}$ and subsequently $\mathcal{N}=\bigcap_{n\in\N}\mathcal{N}_n$. Pick $x\in \pi(\mathcal{N})$ and note that as a special case of Lemma \ref{lemma:quasi-normal}, we have that $\pi(\mathcal{N})$ is a set of full measure. Let $\mathbf{i}\in \mathcal{N}$ be a (not necessarily unique) sequence such that $\pi(\mathbf{i})=x$. Now for any $n\in\N$ there is an integer $k$ such that
\begin{equation*}
    \mathbf{i}=(i_1,\ldots,i_k,\underbrace{i,i,\ldots,i}_n,i_{k+n+1},\ldots).
\end{equation*}
Choose $R_n = r_{\mathbf{i}|_{k}}$ and $r_n=r_{\mathbf{i}|_{k+n}}$, so $\varphi_{\smalli_{k}}(F)\subset B(x,R_n)$, and thus
\begin{equation*}
    \mu(B(x,R_n))\geq\mu(\varphi_{\smalli_{k}}(F))=p_{\mathbf{i}|_{k-1}},
\end{equation*}
and by calculations similar to above,
\begin{equation*}
    \mu(B(x,r_n))\leq MC p_{\mathbf{i}|_{k+n}}.
\end{equation*}
Therefore
\begin{align*}
    \frac{\mu(B(x,R_n))}{\mu(B(x,r_n))}\geq \frac{1}{MC}p_i^{-n}=\frac{1}{MC}(r_i^{-n})^{s}= \frac{1}{MC}\left(\frac{R_n}{r_n}\right)^{s}.
\end{align*}
Since $\frac{R_n}{r_n}\to\infty$ as $n\to \infty$, this shows that $\dim_{\mathrm{A}}(\mu,x)\geq s$. This finishes the proof, since now at $\mu$-almost every $x$, we have $s\leq \dima(\mu,x)\leq \dima\mu\leq s$.
\end{proof}
\begin{huom}
It is an interesting question, if the same formula (\ref{eq:ssc_formula}) for the Assouad dimension of self-similar measures works with even less restrictive separation conditions, such as the weak separation condition.
\end{huom}

\subsection{Self-affine measures on Bedford-McMullen sponges}

A result similar to Theorem \ref{thm:quasi-bernoulli-exact-assouad} also holds for self-affine measures on very strongly separated Bedford-McMullen sponges, which we define as follows. We work in $\R^d$, with $d\geq 2$. Start by choosing integers $n_1<n_2<\ldots<n_d$, and after that choose a subset $\Lambda\subset \prod_{q=1}^d\{0,\ldots,n_q - 1\}$. The set $\Lambda$ is the code space associated with the Bedford-McMullen sponge. For all $\bar{\imath}=(i_1,i_2,\ldots,i_d)\in\Lambda$, we define an affine transform $\varphi_{\bar{\imath}}:[0,1]^d\to[0,1]^d$ by
\begin{equation*}
    \varphi_{\bar{\imath}}(x_1,\ldots,x_d)=\left(\frac{x_1+i_1}{n_1},\ldots,\frac{x_d+i_d}{n_d}\right).
\end{equation*}
The limit set of this IFS is called a \emph{Bedford-McMullen carpet} if $d=2$ or a \emph{Bedford-McMullen sponge} if $d>2$. With this construction, we associate a probability vector $(p_{\bar{\imath}})_{\bar{\imath}\in\Lambda}$, and define the self-affine measure $\mu$ on $F$ as usual. Recall that $\mu$ is related to a Bernoulli measure $\nu$ on the code space $\Sigma$ by (\ref{eq:bernoulli}). To establish bounds for the measures of balls, we need a separation condition which is strictly stronger than the strong separation condition. Following Olsen \cite{O}, we say that a Bedford-McMullen sponge $F$ satisfies the \emph{very strong separation condition (VSSC)}, if for words $(i_1,\ldots,i_d),(j_1,\ldots,j_d)\in\Lambda$ satisfying $i_k=j_k$, for all $k=1,\ldots, q-1$, and $i_q\ne j_q$, for some $q=1,\ldots,d$, we have $|i_q-j_q|>1$.
We also need the following quantity. For $q=1,\ldots,d$ and $\bar{\imath}=(i_1,\ldots,i_d)$, define
\begin{equation}\label{eq:cond_prob}
    p_q(\bar{\imath})=p(i_q|i_1,\ldots,i_{q-1})=\dfrac{\sum\limits_{\substack{\bar{\jmath}\in \Lambda\\ j_k=i_k,\,k=1,\ldots,q}}p_{\bar{\jmath}}}{\sum\limits_{\substack{\bar{\jmath}\in \Lambda\\ j_k=i_k,\,k=1,\ldots,q-1}}p_{\bar{\jmath}}},
\end{equation}
if $(i_1,\ldots,i_q,i_{q+1},\ldots,i_d)\in\Lambda$ for some $i_{q+1},\ldots,i_d$, and 0 otherwise. These numbers can be interpreted as the conditional probabilities that the $q$th digit of a randomly chosen member of $\Lambda$ equals the $q$th digit of $\bar{\imath}$, given that the first $q-1$ coordinates did. The following theorem was proved by Fraser and Howroyd \cite[Theorem 2.6]{FH}.

\begin{thm}\label{thm:bm_formula}
Let $\mu$ be a self-affine measure on a Bedford-McMullen sponge satisfying the VSSC. Then
\begin{equation*}
    \dim_{\mathrm{A}}\mu=\sum_{q=1}^d\max_{\bar{\imath}\in \Lambda}\frac{-\log p_q(\bar{\imath})}{\log n_q}.
\end{equation*}
\end{thm}
Again, we extend this result and prove that the pointwise Assouad dimension coincides with this value at almost every point.

\begin{thm}\label{thm:bm_exact_dimensionality}
Let $\mu$ be a self-affine measure on a Bedford-McMullen sponge $F$ satisfying the VSSC. Then
\begin{equation*}
    \dim_{\mathrm{A}}(\mu,x)=\dim_{\mathrm{A}}\mu,
\end{equation*}
for $\mu$-almost every $x\in F$.
\end{thm}
For the proof we need the concept of \emph{approximate cubes} introduced by Olsen \cite{O}. For clarity, we use $\omega$ to represent members of the set $\Sigma$ instead of $\mathbf{i}$ which we used in the self-similar case. We denote the approximate cube of level $k\in\N$ centered at $\omega=(\bar{\imath}_1,\ldots)=((i_{1,1},\ldots,i_{1,d}),\ldots)\in\Sigma$ by $Q_k(\omega)$, and define it by
\begin{equation*}
    Q_k(\omega)=\{\omega'=(\bar{\jmath}_1,\ldots)\in\Sigma\colon j_{t,q}=i_{t,q},\,\forall q=1,\ldots,d\text{ and }\forall t=1,\ldots L_q(k)\},
\end{equation*}
where $L_q(k)$ is the unique number that satisfies
\begin{equation}\label{eq:approx_cube}
   n_q^{-L_q(k)-1} <n_1^{-k}\leq n_q^{-L_q(k)}.
\end{equation}
The geometric equivalent of the approximate cube $Q_k(\omega)$ is its image under the projection map $\pi:\Sigma\to \R^d$. The image $\pi(Q_k(\omega))$ is contained in
\begin{equation*}
    \prod_{q=1}^d\left[\frac{i_{1,q}}{n_q}+\ldots+\frac{i_{L_q(k),q}}{n_q^{L_q(k)}},\frac{i_{1,q}}{n_q}+\ldots+\frac{i_{L_q(k),q}}{n_q^{L_q(k)}}+\frac{1}{n_q^{L_q(k)}}\right],
\end{equation*}
which is a hypercuboid in $\R^d$ with all side lengths comparable to $n_1^{-k}$.

Olsen \cite{O} observed that the measure of an approximate cube is given by
\begin{equation}\label{eq:appr_cube_msr}
    \mu(\pi(Q_k(\omega)))=\prod_{q=1}^d\prod_{j=0}^{L_q(k)-1}p_q(\sigma^j\omega),
\end{equation}
where $\sigma:\Sigma \to \Sigma$ is the left-shift and $p_q(\omega)=p(i_{1,q}|i_{1,1},\ldots,i_{1,q-1})$, where the right hand side is as in equation (\ref{eq:cond_prob}). Recall also that a Bernoulli measure on the code space $\Sigma$ is shift invariant.

The following proposition by Olsen shows that we can approximate the balls centered at the Bedford-McMullen sponges by approximate cubes of comparable size.

\begin{prop}\label{prop:olsenvssc}
Let $\omega\in\Sigma$ and $k\in\N$.
\begin{enumerate}
    \item If the VSSC is satisfied, then $B(\pi(\omega),2^{-1}n_1^k)\cap F\subset \pi(Q_k(\omega))$.
    \item  $\pi(Q_k(\omega))\subset B(\pi(\omega),(n_1+\ldots+n_d)n_1^k)$.
\end{enumerate}
\end{prop}
The proof of the proposition can be found in \cite[Proposition 6.2.1]{O}. Let us now prove Theorem \ref{thm:bm_exact_dimensionality}. The proof follows ideas of Fraser and Howroyd \cite[Theorem 2.6]{FH}, where they carefully construct sequences of points and scales, which give the desired exponent for the lower bound. The difficulty we face when compared to the approach in \cite{FH}, is that where they have the freedom to choose the point they consider for each pair of scales independently, we have to find a single point where we see the desired behaviour at arbitrarily small scales. Due to the non-conformality of the sponge, this essentially means that we not only need to find long enough sequences of convenient symbols in the symbolic space, but we also have to control their location within the word.

\begin{proof}[Proof of Theorem \ref{thm:bm_exact_dimensionality}]
First we note that $L_q(n)$ increases with $n$ and, since $n_q$ are strictly increasing, decreases with $q$. It is an elementary exercise to show that for every $k\in\N$, there is an integer $n_k$, such that for all $n\geq n_k$, we have
\begin{equation*}
    L_d(n)< L_d(n+k)<L_{d-1}(n)<L_{d-1}(n+k)<\ldots<L_1(n)<L_1(n+k).
\end{equation*}
For $q=1,\ldots,d$, let $p_q^{\min}=\min_{\bar{\imath}\in\Lambda}p_q(\bar{\imath})$, and let $\bar{\imath}_q^{\min}$ be some element of $\Lambda$ which achieves this minimum. Define for every $k\in\N$ the set
\begin{equation*}
    I_k=\bigcup_{n\geq n_k}\bigcap_{q=1}^{d}\sigma^{-L_q(n)}[\underbrace{\bar{\imath}_q^{\min},\ldots,\bar{\imath}_q^{\min}}_{L_q(n+k)-L_q(n)\text{ times}}].
\end{equation*}
Note that an element $\omega\in I_k$ has the form
\begin{align}
    \omega&=(\bar{\imath}_1,\ldots \bar{\imath}_{L_d(n)},\bar{\imath}_d^{\min},\ldots,\bar{\imath}_d^{\min},\bar{\imath}_{L_d(n+k)+1},\ldots,\bar{\imath}_{L_{2}(n)},\label{eq:omegaform}\\
    &\bar{\imath}_2^{\min},\ldots,\bar{\imath}_2^{\min}, \bar{\imath}_{L_2(n+k)+1},\ldots,\bar{\imath}_{L_1(n)},\bar{\imath}_1^{\min},\ldots,\bar{\imath}_1^{\min},\bar{\imath}_{L_1(n+k)+1},\ldots).\nonumber
\end{align}
It is also a simple exercise to show that if $\mathbf{i},\mathbf{j}\in\Sigma_*$, and $q,\ell\in \N$, such that $\ell > q+|\mathbf{i}|$, and $A,B\subset \Sigma$, with $A\subset \Lambda^q\times[\mathbf{i}]$ and $B\subset \Lambda^{\ell}\times[\mathbf{j}]$, then
\begin{equation}\label{eq:independent}
    \nu(A\cap B)=\nu(A)\nu(B).
\end{equation}
Now we choose $m_1=n_k$ and then inductively $m_i=L_1(m_{i-1}+k)+1$, for every $i>1$, and define $A_i\coloneqq\big(\bigcap_{q=1}^{d}\sigma^{-L_q(m_i)}[\underbrace{\bar{\imath}_q^{\min},\ldots,\bar{\imath}_q^{\min}}_{L_q(m_i+k)-L_q(m_i)\text{ times}}]\big)^c$. Noting that $I_k^c\subset\bigcap_{i\in\N}A_i$ and applying (\ref{eq:independent}) inductively first to the sets $A_i$ and then to the sets $\sigma^{-L_q(m_i)}[\underbrace{\bar{\imath}_q^{\min},\ldots,\bar{\imath}_q^{\min}}_{L_q(m_i+k)-L_q(m_i)\text{ times}}]$, we obtain
\begin{align*}
    \nu(I_k^c)&\leq\nu\left(\bigcap_{i\in\N}A_i\right)= \prod_{i\in\N}\nu(A_i)=\prod_{i\in\N}\Bigg(1-\nu\bigg(\bigcap_{q=1}^{d}\sigma^{-L_q(m_i)}[\underbrace{\bar{\imath}_q^{\min},\ldots,\bar{\imath}_q^{\min}}_{L_q(m_i+k)-L_q(m_i)\text{ times}}]\bigg)\Bigg)\\
    &=\prod_{i\in\N}\left(1-\prod_{q=1}^{d}(p_q^{\min})^{L_q(m_i+k)-L_q(m_i)}\right)\leq \prod_{i\in\N}\Bigg(\underbrace{1-(p_q^{\min})^d}_{<1}\Bigg)=0.
\end{align*}
Thus $\nu(I_k)=1$, and moreover
$\nu(I)=1$, where $I=\bigcap_{k\in\N}I_k$.

Now let $s=\dima\mu$ given by Theorem \ref{thm:bm_formula}, $x=\pi(\omega)$, where $\omega\in I$, and let $R_k=(n_1+\ldots+n_d)n_1^{-n-1}$, and $r_k=2^{-1}n_1^{-(n+k)-1}$, where $k$ and $n$ are chosen, such that $\omega$ is given by equation (\ref{eq:omegaform}). Observe that by Proposition \ref{prop:olsenvssc} and equations (\ref{eq:approx_cube}) and (\ref{eq:appr_cube_msr}), we have
\begin{align*}
    \frac{\mu(B(x,R_k))}{\mu(B(x,r_k))}&=\frac{\prod_{q=1}^d\prod_{j=0}^{L_q(n)}p_q(\sigma^j\omega)}{\prod_{q=1}^d\prod_{j=0}^{L_q(n+k)}p_q(\sigma^j\omega)}=\frac{1}{\prod_{q=1}^d\prod_{j=L_q(n)-1}^{L_q(n+k)}p_q(\sigma^j\omega)}\\
    &=\prod_{q=1}^d\left(\frac{1}{p_q^{\min}}\right)^{L_q(n+k)-L_q(n)+2}\geq \prod_{q=1}^d\left(\frac{1}{p_q^{\min}}\right)^{(n+k)\frac{\log n_1}{\log n_q}-n\frac{\log n_1}{\log n_q}+1}\\
    &\geq (p_q^{\min})^{-d}\prod_{q=1}^d\left(\frac{1}{p_q^{\min}}\right)^{k\frac{\log n_1}{\log n_q}}=(p_q^{\min})^{-d}\prod_{q=1}^d\left(n_1^k\right)^{\frac{-\log p_q^{\min}}{\log n_q}}\\
    &\geq (\min_{q}p_q^{\min})^{-d}\left(n_1^k\right)^{s}=C\left(\frac{R_k}{r_k}\right)^{s},
\end{align*}
where $C=(\min_{q}p_q^{\min})^{-d}\cdot(2(n_1+\ldots+n_d))^s>0$ is a constant. Taking $k\to\infty$, we see that $\frac{R_k}{r_k}\to \infty$, which is enough to prove that $\dim_{\mathrm{A}}(\mu,x)\geq s$. This holds for all $x=\pi(\omega)$, such that $\omega\in I$, where $I$ has full measure, proving the claim.
\end{proof}

\begin{example}\label{ex:bm_example}
    Here we give an example of a measure $\mu$, with $\updim_{\mathrm{M}}\mu<\dima(\mu,x)$, for $\mu$-almost every $x$. Let $\mu$ be a self-affine measure on a Bedford-McMullen carpet. By \cite[Theorem 8.6.2]{F}, the upper Minkowski dimension of $\mu$ is given, in the notation of Section \ref{sec:ssm}, by the formula
    \begin{equation*}
        \updim_{\mathrm{M}}\mu=\max_{\bar{\imath}\in\Lambda}\left(\frac{-\log p_{\bar{\imath}}}{\log n_2}\right)+\max_{\bar{\imath}\in\Lambda}\left(\frac{\log p_1(\bar{\imath})}{\log n_2}+\frac{-\log p_1(\bar{\imath})}{\log n_1}\right),
    \end{equation*}
    and from Theorems \ref{thm:bm_formula} and \ref{thm:bm_exact_dimensionality}, it follows that the pointwise Assouad dimension is given by
    \begin{equation*}
        \dima(\mu,x)=\max_{\bar{\imath}\in\Lambda}\left(\frac{-\log p_{\bar{\imath}}}{\log n_2}+\frac{\log p_1(\bar{\imath})}{\log n_2}\right)+\max_{\bar{\imath}\in\Lambda}\left(\frac{-\log p_1(\bar{\imath})}{\log n_1}\right),
    \end{equation*}
    at $\mu$-almost every $x$. By choosing the $p_{\bar{\imath}}$, for example in a way that $p_{\bar{\imath}}$ and $\frac{p_{\bar{\imath}}}{p_1(\bar{\imath})}$ are minimized in the same column, and $p_1(\bar{\imath})$ is minimized in a different column, we have
    \begin{equation*}
        \dima(\mu,x)>\updim_{\mathrm{M}}\mu,
    \end{equation*}
    for $\mu$-almost every $x$. For example, we may choose $n_1=3$ and $n_2=4$, and $\Lambda=\{(0,0),(0,3),(2,0)\}$, with $p_{(0,0)}=\frac{1}{8},p_{(0,3)}=\frac{5}{8}$ and $p_{(2,0)}=\frac{1}{4}$. Then we have
    \begin{equation*}
        \dima(\mu,x)=\frac{\log 6}{\log 4}+\frac{\log 4}{\log 3}> \frac{\log2}{\log 4}+\frac{\log 4}{\log 3}=\updim_{\mathrm{M}}\mu.,
    \end{equation*}
    for $\mu$-almost every $x$.
    \end{example}

\section{Discussion}
Most of the results of this paper follow a similar pattern by providing exact dimensionality properties for the pointwise Assouad dimension. A natural follow up to the results of this paper would be to conduct finer analysis of the pointwise Assouad dimension and develop tools for \emph{multifractal analysis} in this setting. Classically, the multifractal spectrum of a measure is given by the Hausdorff dimension of $\alpha$-level sets of the local dimension.
The celebrated multifractal formalism states that, in many cases, this spectrum is given by the Legendre transform of the $L^q$-spectrum of the measure, see e.g. Chapter 11 of \cite{Falc1} for details. Of course, a natural question to ask is if something similar is true for the dimension spectrum of the level sets of the pointwise Assouad dimension.

\begin{quest}
What is the multifractal Assouad spectrum of a strongly separated self-similar measure $\mu$? By this we mean quantity
\begin{equation*}
     f_{\mathrm{A}}(\alpha)\coloneqq \dim_{\mathrm{H}}\{x\in X\colon \dima(\mu,x)=\alpha\}.
\end{equation*}
\end{quest}
Using the Hausdorff dimension instead of the Assouad dimension in the definition is natural, since it is easy to see that each $\alpha$-level set of the pointwise Assouad dimensions is dense in the support and the Assouad dimension of sets is stable under closures.

\subsection*{Acknowledgements}
I would like to thank Antti Käenmäki and Ville Suomala for many fruitful conversations on the contents of the paper and Balázs Bárány who introduced me to the concept of invariant measures with place dependent probabilities during his visit at the University of Oulu. I would also like to thank the anonymous referee for carefully examining the manuscript and providing many detailed comments. Finally, I thank the referee, and Anders and Jana Björn for pointing out the reference \cite{BBL}. This research was partially funded by the Magnus Ehrnrooth foundation.

\bibliography{references}

\providecommand{\bysame}{\leavevmode\hbox to3em{\hrulefill}\thinspace}
\providecommand{\MR}{\relax\ifhmode\unskip\space\fi MR }
\providecommand{\MRhref}[2]{%
  \href{http://www.ams.org/mathscinet-getitem?mr=#1}{#2}
}
\providecommand{\href}[2]{#2}
\begin{thebibliography}{10}

\bibitem{A}
P.~Assouad, \emph{Espaces métriques, plongements, facteurs}, Thèse de
  doctorat d’État, Univ. Paris XI, 91405 Orsay, 1977.

\bibitem{Barn}
M.~F. Barnsley, S.~G. Demko, J.~H. Elton, and J.~S. Geronimo, \emph{Invariant
  measures for {Markov} processes arising from iterated function systems with
  place-dependent probabilities}, Ann. I. H. Poincare-Pr. \textbf{24} (1988),
  no.~3, 367--394.

\bibitem{BBL}
A.~Björn, J.~Björn, and J.~Lehrbäck, \emph{Sharp capacity estimates for
  annuli in weighted {$\R^n$} and in metric spaces}, Math. Z. \textbf{286}
  (2017), 1173–1215.

\bibitem{Bow}
R.~Bowen, \emph{Equilibrium states and the ergodic theory of {Anosov}
  diffeomorphisms}, Springer-Verlag, 1975.

\bibitem{BHR}
B.~Bárány, M.~Hochman, and A.~Rapaport, \emph{Hausdorff dimension of planar
  self-affine sets and measures}, Invent. Math. \textbf{216} (2019), 601--659.

\bibitem{Falc1}
K.~Falconer, \emph{Techniques in fractal geometry}, John Wiley \& Sons Ltd,
  1996.

\bibitem{KFF}
K.~J. Falconer, J.~M. Fraser, and A.~Käenmäki, \emph{Minkowski dimension for
  measures}, preprint (2020).

\bibitem{Fan}
A.~H. Fan and K.-L. Lau, \emph{Iterated function system and {Ruelle} operator},
  J. Math. Anal. Appl. \textbf{231} (1999), 319--344.

\bibitem{F}
J.~M. Fraser, \emph{Assouad dimension and fractal geometry}, Cambridge Tracts
  in Mathematics, Cambridge University Press, 2020.

\bibitem{FH}
J.~M. Fraser and D.~C. Howroyd, \emph{On the upper regularity dimensions of
  measures}, Indiana Univ. Math. J. \textbf{69} (2020), 685--712.

\bibitem{Heur}
Y.~Heurteaux, \emph{Estimations de la dimension inférieure et de la dimension
  supérieure des mesures}, Ann. I. H. Poincare-Pr. \textbf{34} (1998),
  309--338.

\bibitem{Hoch}
M.~Hochman, \emph{On self-similar sets with overlaps and inverse theorems for
  entropy}, Ann. Math. \textbf{180} (2014), 773–82.

\bibitem{HR}
M.~Hochman and A.~Rapaport, \emph{Hausdorff dimension of planar self-affine
  sets and measures with overlaps}, J. Eur. Math. Soc. (2021), published online
  first.

\bibitem{H}
J.~E. Hutchinson, \emph{Fractals and self-similarity}, Indiana Univ. Math. J.
  \textbf{30} (1981), 713--747.

\bibitem{KL}
S.~Keith and T.~Laakso, \emph{Conformal {Assouad} dimension and modulus}, Geom.
  Funct. Anal. \textbf{14} (2004), 1278--1321.

\bibitem{KLV}
A.~Käenmäki, J.~Lehrbäck, and M.~Vuorinen, \emph{Dimensions, {Whitney}
  covers, and tubular neighborhoods}, Indiana Univ. Math. J. \textbf{62}
  (2013), 1861--1889.

\bibitem{Kig}
J.~Kigami, \emph{Analysis on fractals}, Cambridge University Press, 2001.

\bibitem{LR}
E.~Le~Donne and T.~Rajala, \emph{{Assouad} dimension, {Nagata} dimension, and
  uniformly close metric tangents}, Indiana Univ. Math. J. \textbf{64} (2015),
  21--54.

\bibitem{LS}
J.~Luukkainen and E.~Saksman, \emph{Every complete doubling metric space
  carries a doubling measure}, Proc. Am. Math. Soc. \textbf{126} (1998), no.~2,
  531--534.

\bibitem{MT}
{J. M.} Mackay and {J. T.} Tyson, \emph{Conformal dimension: Theory and
  application}, University Lecture Series, American Mathematical Society,
  United States, 2010 (English (US)).

\bibitem{MU}
R.~D. Mauldin and M.~Urba\'{n}ski, \emph{Dimensions and measures in infinite
  iterated function systems}, P. Lond. Math. Soc. \textbf{73} (1996), 105--73.

\bibitem{O}
L.~Olsen, \emph{Self-affine multifractal {Sierpinski} sponges in {$R^d$}},
  Pacific J. Math \textbf{183} (1998), 143--199.

\bibitem{T}
S.~Troscheit, \emph{On quasisymmetric embeddings of the {Brownian} map and
  continuum trees}, Probab. Theor. Rel. \textbf{179} (2021), 1023--1046.

\bibitem{VK}
A.~L. Vol'berg and S.~V. Konyagin, \emph{There is a homogeneous measure on any
  compact subset in {$\R^n$}}, Soviet Math. Dokl. \textbf{30} (1984), 453--456
  (Russian).

\bibitem{Yung}
P.-L. Yung, \emph{Doubling properties of self-similar measures}, Indiana Univ.
  Math. J. \textbf{56} (2007), 965--990.

\end{thebibliography}
\bibliographystyle{amsplain}

\end{document}